\documentclass[10pt]{amsart}
 \font\sr = msbm10 scaled \magstep 1
\newcommand{\Q}{\mbox{\sr Q}}
\newcommand{\R}{\mbox{\sr R}}

\newtheorem{thm}{Theorem}[section]
\newtheorem{lemma}[thm]{Lemma}
\newtheorem{defn}[thm]{Definition}
\newtheorem{nota}[thm]{Notation}
\newtheorem{rmk}[thm]{Remark}

\newtheorem{coro}[thm]{Corollary}
\newtheorem{examp}[thm]{Example}
\newtheorem{prop}[thm]{Proposition}
\newtheorem{assumptions}[thm]{Assumption}

\newcommand{\bpr}{\begin{proof}}
\newcommand{\epr}{\end{proof}}
\newcommand{\bt}{\begin{thm}}
\newcommand{\et}{\end{thm}}
\newcommand{\bl}{\begin{lemma}}
\newcommand{\el}{\end{lemma}}
\newcommand{\br}{\begin{rmk}}
\newcommand{\er}{\end{rmk}}
\newcommand{\bc}{\begin{coro}}
\newcommand{\ec}{\end{coro}}
\newcommand{\bd}{\begin{defn}}
\newcommand{\ed}{\end{defn}}
\newcommand{\bn}{\begin{nota}}
\newcommand{\en}{\end{nota}}
\newcommand{\bex}{\begin{examp}}
\newcommand{\eex}{\end{examp}}
\newcommand{\bp}{\begin{prop}}
\newcommand{\ep}{\end{prop}}
\newcommand{\ba}{\begin{assumptions}}
\newcommand{\ea}{\end{assumptions}}
\newcommand{\beq}{\begin{equation}}
\newcommand{\eeq}{\end{equation}}

\newcommand{\ds}{\displaystyle}

\begin{document}
\renewcommand{\baselinestretch}{0.9}

\date{\today}
{

%\normalsize

\title[Strings with moving ends]{A remark on observability of the wave equation with moving boundary}

\author{Ka\"{\i}s Ammari}
\address{UR Analysis and Control of PDEs, UR13ES64, Department of Mathematics, Faculty of Sciences of Monastir,
University of Monastir, 5019 Monastir, Tunisia}
\email{kais.ammari@fsm.rnu.tn}

\author{Ahmed Bchatnia}
\address{Department of
Mathematics, Faculty of Sciences of Tunis, University of Tunis El
Manar, Tunisia} \email{ahmed.bchatnia@fst.rnu.tn}

\author{Karim El Mufti}
\address{UR Analysis and Control of PDEs, UR13ES64, ISCAE, University of Manouba, Tunisia}
\email{karim.elmufti@iscae.rnu.tn}

\begin{abstract}
We deal with the wave equation with assigned moving boundary
($0<x<a(t)$) upon which Dirichlet or mixed boundary conditions are
specified, here $a(t)$ is assumed to move slower than the light and
periodically. Moreover $a$ is continuous, piecewise linear with two
independent parameters. Our major concern will be an observation
problem which is based measuring, at each $t>0$ of the transverse
velocity at $a(t)$. The key to the results is the use of a reduction
theorem \cite{Yoc}.
\end{abstract}
\date{}

\maketitle

\tableofcontents

%\vfill\break

\section{Introduction and main results}
We consider the following problems:
\begin{equation}\label{B.Khiar}
\left\{\begin{array}{lll} u_{tt}-u_{xx}=0 \;\;\;\;
\mbox{for} \;\;\;\; 0<x<a(t), \, t > 0,\\
u(x,0)=\phi(x), u_t(x,0)=\psi(x), \;\; 0<x<a(0), \\
\end{array}\right.
\end{equation}
$(\phi,\psi)\in H^1((0,a(0)))\times L^2((0,a(0)))$, with Dirichlet
boundary conditions
\begin{equation}\label{dir}
u(0,t)=0 \;\;\;\; \mbox{and} \;\;\;\; u(a(t),t)= 0, \, t > 0,
\end{equation}
or with mixed boundary conditions for which (\ref{dir}) is replaced
by
\begin{equation}\label{neu}
u(0,t)=0 \;\;\;\; \mbox{and} \;\;\;\; u_x(a(t),t)=0, \, t > 0,
\end{equation}
the subscripts denote partial differentiations, here $a$ is a
strictly positive real function which is continuous, periodic,
piecewise linear.\\
Our major concern will
be to find the associated curves $a(t)$ for which,
\begin{equation}\label{obd}
\int_{0}^{T} \left|u_x(a(t),t)\right|^2 \, dt \geq C^{*} \,
\left(\|\phi\|^2_{H^1_0(0,a(0))} + \|\psi\|^2_ {L^{2}(0,a(0))} \right),
\end{equation}
for (\ref{B.Khiar})-(\ref{dir}) and
\begin{equation}\label{obn}
\int_{0}^{T} \left|u_t(a(t),t)\right|^2 \, dt \geq C^{*} \,
\left(\|\phi\|^2_{H^1_l(0,a(0))} + \|\psi\|^2_ {L^{2}(0,a(0))} \right),
\end{equation}
for (\ref{B.Khiar})-(\ref{neu}) are valid,  where $H_l^1(0,a(0)) = \left\{ f \in
H^1(0,a(0)) \mbox{ such that } f(0)=0\right\}.$ 
Note that if $a$ is a constant, the
observability inequality 
\begin{equation}\label{Obser1}
\int_{0}^{T} \left|u_x(a,t)\right|^2 \, dt \geq C^{*} \,
\left(\|\phi\|^2_{H^1_0(0,a)} + \|\psi\|^2_ {L^{2}(0,a)} \right)\;
\mbox{for some positive constant} \; C^{*},
\end{equation}
holds if $T\geq 2a$ for the Dirichlet problem. Also,
\begin{equation}\label{Obser2}
\int_{0}^{T} \left|u_t(a,t)\right|^2 \, dt \geq C^{*} \,
\left(\|\phi\|^2_{H^1_l(0,a)} + \|\psi\|^2_ {L^{2}(0,a)} \right),
\end{equation}
holds for the mixed problem.
\bigskip 

In \cite{C} the author consider the system
\begin{equation*}
\left\{\begin{array}{lll} \varphi_{tt}-\varphi_{xx}=0 \;\;\;\;
\mbox{for} \;\;\;\; 0<x<1, \, t > 0,\\
\varphi(0,t) = \varphi(1,t) = 0, \, t > 0, \\
\varphi(x,0)=\phi(x), \varphi_t(x,0)=\psi(x), \;\; 0<x<1. \\
\end{array}\right.
\end{equation*}

For a suitable class of curves $a(t)$, which are $a : [0,T] \rightarrow (0,L)$ in the class $C^1([0, T])$ piecewise, i.e. there exists a partition of $[0, T], 0 = t_0 < t_1 < \cdots < t_n = T,$ such that $a\in C^1([t_i, t_{i+1}])$ for all $i = 0, \cdots, n-1.$ Assume also that this partition can be chosen in such a way that $1-|a'(t)|$ does not change the sign in $t\in   [t_i, t_{i+1}],$ for all $i = 0, 1, \cdots , n-1.$ Also he makes the following hypothesis:
\begin{enumerate}
\item There exists constants $c_1, c_2 > 0$ and a finite number of open subintervals $I_j \subset [0, T]$ with
$j = 0, \cdots , J$ such that, for each subinterval $I_j$, $a\in C^1(I_j)$ and it satisfies the following two
conditions :\\
$\bullet$ $c_1\leq |a'(t)| \leq c_2$ for all $t \in I_j$,\\
$\bullet$ $1 -|a'(t)|$ does not change the sign in $t\in I_j$. \\
We assume, without loss of generality, that there exists $j_1$ with $-1 \leq  j_1 \leq J$ such that $a(t)$ is decreasing in $I_j$ for $0 \leq j \leq j_1$, and $a(t)$ is increasing in $I_j$ for $j_1 < j \leq J.$\\
The case $j_1 = -1$ corresponds to that where $a(t)$ is increasing in all the subintervals $I_j$. Analogously, $j_1 = J$ corresponds to the case where $a(t)$ is decreasing in all the subintervals $I_j$.
\item For each $j = 0, \cdots , J$, let $U_j$ be the subintervals defined as follows:
\begin{equation*}
U_j =  \left\{ \begin{array}{ll}  \{ s-a(s) \mbox{ with }  s\in I_j \}   \mbox{ if } j\leq j_1   \\
\{ s+a(s) \mbox{ with }  s\in I_j \}   \mbox{ if } j> j_1 . \end{array} \right.
\end{equation*}
Then, there exists an interval $W_1$ with length $(W_1) > 2L$ such that
$$
W_1 \subset \overline{\overset{J}{\underset{j=0}{\cup}}U_j.}
$$
\item For each $j = 0, \cdots , J$, let $V_j$ be the subintervals defined as follows:
\begin{equation*}
V_j =  \left\{ \begin{array}{ll}  \{ s+a(s) \mbox{ with }  s\in I_j \}   \mbox{ if } j\leq j_1   \\
\{ s-a(s) \mbox{ with }  s\in I_j \}   \mbox{ if } j> j_1 . \end{array} \right.
\end{equation*}
\end{enumerate}
He gets the following observability estimate:
\begin{equation}\label{cas}
\int_{0}^{T} \left|\frac{d}{dt} [\varphi(a(t),t)]\right|^2 \, dt \geq C^{*} \,
\left(\|\phi\|^2_{H^1_0(0,1)} + \|\psi\|^2_ {L^{2}(0,1)} \right),
\end{equation}
where $T$ is given by an optical geometric condition requiring that any ray, starting anywhere in the
domain and with any initial direction, must meet the dissipation zone before the time $T$.
Also, he gives several examples of curves for exact controllability related to the following system:
\begin{equation*}
\left\{\begin{array}{lll} u_{tt}-u_{xx}=f(t)\delta_{a(t)} \;\;\;\;
\mbox{for} \;\;\;\; 0<x<1, \, t > 0,\\
u(0,t) = u(1,t) = 0, \, t > 0, \\
u(x,0)=\phi(x), u_t(x,0)=\psi(x), \;\; 0<x<1, \\
\end{array}\right.
\end{equation*}
 via (\ref{cas}).
 
\bigskip 
Here we introduce a new approach that provides (\ref{obd}) and (\ref{obn}) for another class of curves $a(t)$ with less of regularity. 
We start with some notations and known results.  

\bigskip 
Let $\mbox{Lip} (\R)$ be the space
of Lipschitz continuous functions on $\R$. We shall denote the
Lipschitz constant of a function $F$ by
$$
L(F):=\sup_{x,y\in \R,x\neq y} \left | \frac{F(x)-F(y)}{x-y}\right
|.
$$
Furthermore, denote by $D_{p}$ the set of functions continuous and
strictly increasing of the form $x+g(x)$, where $g(x)$ is a periodic
continuous function.
\begin{prop}
Let $a$ be a periodic function. Then
\begin{equation}\label{Fexpr}
F:=(I+a)\circ (I-a)^{-1}
\end{equation}
belongs to $D_{p}$. Moreover, the rotation number $\rho(F)$ defined
by
$$
\rho(F)= \ds \lim_{n \rightarrow \infty} \frac{F^n(x)-x}{n}
$$
exists, and the limit is equal for all $x \in \R$.
\end{prop}
Rigorous studies pointing out the use of rotation numbers has led to
fruitful contributions, one of which is an elegant and important result
(see \cite[section II]{Her} for more details):

\bigskip 
Assume that $a(t)$ is a periodic function, $a(t)>0$, $a \in
\mbox{Lip}(\R)$ such that $L(a)\in [0,1)$. Assume also that
$|a'(t)|<1$ for all $t \in \R$ and $\rho (F)\in\R\setminus\Q$ such
that there exists a function $H \in D_{p}$ and
\begin{equation}\label{redF}
H^{-1} \circ F \circ H(\xi) = \xi + \rho(F).
\end{equation}
Before stating our main results, let us specify some hypotheses on
$H$.
\begin{assumptions}\label{a1}
There exist $\lambda_1>0$ and $\lambda_2>0$ such that
\begin{equation}
\lambda_1 \leq H^\prime (t) \leq \lambda_2, \, t \in \mathbb{R}.
\end{equation}
\end{assumptions}
\begin{assumptions}\label{a2}
The function $\displaystyle
b(t):=\frac{H'(a(t)+t)-H'(-a(t)+t)}{H'(a(t)+t)+H'(-a(t)+t)}$
satisfies
\begin{equation}
c_1 \leq b(t)   \leq c_2, \;\; c_1,c_2>0, \mbox{ for all } t\in \mathbb{R}.
\end{equation}
\end{assumptions}
\begin{rmk} We make less assumptions  and get on the occasion a larger class of functions $a(t)$ in connexion with the work of Castro \cite{C}.
\end{rmk}
We give an example where assumptions \ref{a1} and \ref{a2} are
guaranteed as in \cite{Gon1}.

\smallskip 

Let $a$ be continuous and
periodic on $\R$, $a>0$, be such that
\[
a(t):=\left\{\begin{array}{ll} \alpha
t+\frac{\alpha(1-\alpha)(1+\beta)}{2(\alpha-\beta)} & \mbox{if
$\frac{\alpha(1+\beta)}{2(\alpha-\beta)}\leq t\leq
\frac{\alpha(1+\beta)-2\beta}{2(\alpha-\beta)},$}\\
\beta t-\beta+\frac{\alpha(1-\beta^2)}{2(\alpha-\beta)} & \mbox{if
$\frac{\alpha(1+\beta)-2\beta}{2(\alpha-\beta)} \leq
t\leq\frac{\alpha(3+\beta)-2\beta}{2(\alpha-\beta)}$,}
\end{array}
\right.
\]
with $\alpha$, $\beta\in (-1,1)$. Let
$l_1:=\frac{1+\alpha}{1-\alpha}$, $l_2:=\frac{1+\beta}{1-\beta}$.
The definition of $a$ is choosen such that $F$ is directly given on
$[0,1)$ and we extend $F$ through the formula: $F(x+1)=F(x)+1$ for
any $x\in\R$. The function $F$ is defined by :
\[
F(x):=(I+a)\circ (I-a)^{-1}(x)=\left\{\begin{array}{ll}l_1x+F_0
& \mbox{if $0\leq x\leq x_0$}\\
l_2x+F_0+1-l_2 & \mbox{if $x_0<x<1,$}
\end{array}
\right.
\]
with $F_0:=\frac{l_2(l_1-1)}{l_1-l_2}$,
$x_0:=\frac{1-l_2}{l_1-l_2}$.
\smallskip 
Also the rotation number is given by the expression:
\begin{equation}
\label{formrho} \rho (F)=\frac{\ln l_1}{\ln
\left(\frac{l_1}{l_2}\right)},
\end{equation}
and the function $H$ given by (\ref{redF}) is done
by 
$$H(x)=h_0\ln \left(|x+h_1| \right) + h_2,$$ where
$h_0=\frac{1}{\ln\left(\frac{l_1}{l_2}\right)}$, $h_1=\frac{l_2}{l_1-l_2}$ and
$h_2=-\ln \left(|h_1| \right),$ and satisfies the following inequalities: if $l_1>l_2$,
\begin{equation}\label{l2l1}
\frac{1}{\ln( \frac{l_1}{l_2})}\frac{l_1-l_2}{l_1}\leq H'(x)\leq
\frac{1}{\ln( \frac{l_1}{l_2})}\frac{l_1-l_2}{l_2},
\end{equation}
and if $l_1<l_2$,
\begin{equation}\label{l1l2}
\frac{1}{\ln( \frac{l_1}{l_2})}\frac{l_1-l_2}{l_2}\leq H'(x)\leq
\frac{1}{\ln( \frac{l_1}{l_2})}\frac{l_1-l_2}{l_1}\cdot
\end{equation}
The function $b$ which is 1-periodic is defined on $[0,1)$ by
$$\displaystyle{b(t)=-\frac{a(t)}{t+\frac{l_2}{l_1-l_2}}\cdot}$$
Assuming that $l_1<l_2,$ this function satisfies for all $t \in \mathbb{R}$
\begin{equation}\label{btau}
\frac{a_{min}(l_2-l_1)}{l_2}\leq\frac{a(t)(l_2-l_1)}{l_2}\leq
b(t)\leq\frac{a(t)(l_2-l_1)}{l_1}\leq\frac{a_{max}(l_2-l_1)}{l_1}.
\end{equation}

On the existence of solutions to the Dirichlet or the mixed problem,
we refer the reader to \cite{gon}. We have the following
proposition:

\begin{prop}
If $a\in\mbox{Lip} (\R)$, $L(a)\in [0,1)$, $a>0$ and 
$$
(\varphi_0,\varphi_1)\in H_0^1((0,a(0)))\times L^2((0,a(0))),\; \hbox{or in} \; H_l^1((0,a(0)))\times L^2((0,a(0))),
$$
denote by $Q:=(0,a(t))\times \R_{+}$ and $Q_{\tau}:=(0,a(t))\times
(0,\tau), \tau \in \R_{+}$. There exists a unique weak solution\footnote{$u\in H^1(Q_{\tau})$ is called a weak solution of either the Dirichlet or the mixed problem if $u_{tt}-u_{xx}=0$ in ${\mathcal D}^\prime(Q)$ and the boundary conditions are satisfied.}$u$
of either the Dirichlet or the mixed problem satisfying the initial
conditions $u(x,0)=\phi(x), u_t(x,0)=\psi(x) \;\; 0<x<a(0)$. Moreover
there exists $f\in H_{\mbox{loc}}^1(\R)\cap L^\infty (\R)$ such that
\begin{equation}
u(x,t)=f(t+x)-f(t-x)\quad\mbox{a.e. in $Q$},
\end{equation}
and $u\in L^{\infty}(Q)\cap H^1(Q_{\tau})$.
\end{prop}
Our main results are stated as follows:
\begin{thm}[Neumann observability]\label{dirichlet}
Under the assumption \ref{a1}, there exist $T, C^{*}> 0$ such that
for all $u$ solution of the system (\ref{B.Khiar}) with the
Dirichlet boundary condition (\ref{dir}) and initial data
$(\phi,\psi) \in H^1_0(0,a(0)) \times L^2(0,a(0)),$ we have
\begin{equation}\label{Obserneu}
\int_{0}^{T} \left|\ u_{x}(a(t),t)\right|^2 \, dt \geq C^{*} \,
\left(\|\phi\|^2_{H^1_0(0,a(0))} + \|\psi\|^2_{L^2(0,a(0))} \right).
\end{equation}
\end{thm}
\begin{rmk}
We give a similiar result in the appendix, which concern the Dirichlet observability.
\end{rmk}
The exact controllability problem for the system 
\begin{equation}\label{cont}
\left\{\begin{array}{lll} 
u_{tt}-u_{xx}=0 \;\;\;\;
\mbox{for} \;\;\;\; 0<x<a(t), \, t > 0,\\
u(0,t)=0 \;\;\;\; \mbox{and} \;\;\;\; u(a(t),t)=r(t), \, t > 0,\\
u(x,0)=\phi(x), u_t(x,0)=\psi(x) \;\; 0<x<a(0) \\
\end{array}\right.
\end{equation}
at time $T$ is the following: for each $(\phi, \psi) \in L^2(0,a(0))\times H^{-1}(0,a(0)),$ find
$r \in L^2(0,T)$ such that the corresponding solution to (\ref{cont}) satisfies
$u( .,T) = 0, u_t(.,T) = 0$  in $(0, a(T))$. 

\bigskip 

Based on the observability estimate mentioned above, we get:
\begin{coro}\label{cor}
Assume that $l_1 > l_2$, then there exist $T>0$ and $r \in L^2(0,T)$ such that the system (\ref{cont}) is exactly controllable at time $T$.
\end{coro}

\bigskip
The paper is organized as follows: In section \ref{sec2} we prove our main
results and in the last section we give further comments on the
quasi periodic case.

\section{Proof of the main results} \label{sec2}
We shall construct a transformation of the time-dependent domain
$[0, a(t)] \times \R$ onto $[0, \rho(F) / 2] \times \mathbb{R}$ that
preserves the D'Alembertian form of the wave equations. This
preserving property will reveal very important. Using $H$ given by
(\ref{redF}), we define a domain transformation $\Phi : \mathbb{R}^2
\rightarrow \mathbb{R}^2$ as follows:
\begin{equation}\label{trans}
\left\{\begin{array}{lll}
\xi = (H(x + t) - H(-x + t))/ 2,\\
\tau = (H(x + t) + H(-x + t))/ 2, \\
\end{array}\right.
\end{equation}
for $(x,t) \in \mathbb{R}^2$.

\begin{rmk}
The following propositions can essentially be found in \cite{Ya 5}
(see also the references therein), we reproduce them here for the
reader's convenience and because our presentation is synthetic.
\end{rmk}

\begin{prop}
The transformation $\Phi$ is a bijection of $[0, a(t)] \times
\mathbb{R}$ to $[0, \rho(F)  / 2] \times \R$ and $\Phi$ maps the
boundaries $x=0$ and $x=a(t)$ onto the boundaries $\xi=0$ and
$\xi=\rho(F)  / 2$ (resp).
\end{prop}
\begin{prop}
Let $u(x,t)$ satisfying $(\partial_{t}^{2} - \partial_{x}^{2})
u(x,t)=0$ and $V(\xi,\tau)$ defined by $u (\Phi^{-1}(\xi,\tau))$.
Then the following identity holds
$$
(\partial_{t}^{2} - \partial_{x}^{2}) u(x,t) = K(\xi,\tau)
(\partial_{\tau}^{2} - \partial_{\xi}^{2})V(\xi,\tau)
$$
where $K(\xi,\tau)$ is defined by
$$
4 H' \circ H^{-1}(\xi + \tau) H' \circ H^{-1}(- \xi + \tau) \circ
H^{-1}(\xi + \tau).
$$
\end{prop}

The next lemma will be very useful for the proof of our main
results.
\begin{lemma}\label{second}
Denote by
$$
E_u(t)=\frac{1}{2}\int_{0}^{a(t)}\left[\left|u_t(x,t)\right|^2 +
\left|u_x (x,t)\right|^2 \right] \, dx
$$
the energy of the field $u$, and
$$
E_V(\tau)=\displaystyle\int^{\rho(F) /2}_{0} \left(\left|V_{\xi}
(\xi,\tau)\right|^2 + \left|V_{\tau}(\xi,\tau)\right|^2\right) \,
d\xi,
$$
the energy of the field $V$. There are two positive constants $C_1$
and $C_2$ such that
\begin{equation}
C_1 E_V(\tau)\leq E_u(t)\leq C_2 E_V(\tau).
\end{equation}
\end{lemma}
\begin{proof} We calculate,
$$
\partial_{t}u = \partial_{\xi}V \partial_{t}\xi
+ \partial_{\tau}V \partial_{t}\tau \;\;\; \mbox{and} \;\;\;
\partial_{x}u = \partial_{\xi}V \partial_{x}\xi +
\partial_{\tau}V \partial_{x}\tau
$$
and so,
$$
E_u(t)= \frac{1}{2}\int_{0}^{a(t)} \left[\left|u_t(x,t)\right|^2 +
\left|u_x (x,t)\right|^2 \right] \, dx =
$$
$$
\frac{1}{2}\int_{0}^{a(t)} \left\{[V_{\xi} \xi_{t} + V_{\tau}
\tau_{t}]^2 + \left|V_{\xi} \xi_{x} + V_{\tau} \tau_{x}\right|^2
\right\} \, dx .
$$
Make use of:
$$
\xi_{x} = (\partial_{x}\xi) = (\partial_{t}\tau) = \tau_{t}= [H'(x +
t) + H'(-x + t)] / 2,
$$
$$
\xi_{t} = (\partial_{t}\xi) = (\partial_{x}\tau) = \tau_{x} = [H'(x
+ t) - H'(-x + t)] / 2.
$$
Hence,
$$ \xi_{t}^2 + \xi_{x}^2 = \tau_{t}^2 + \tau_{x}^2
= \frac{1}{2}\left[\left|H'(x + t)\right|^2 + \left|H'(-x +
t)\right|^2 \right],
$$
$$
\xi_{t}\tau_{t} = \xi_{x}\tau_{x} = \frac{1}{4}[(H'(x + t))^2 -
(H'(-x + t))^2].
$$
Back to the energy,
$$
E_u(t)= \int_{0}^{a(t)} \frac{1}{4}\left[\left|V_{\xi}\right|^2 +
\left|V_{\tau}\right|^2 \right] \left[\left|H'(x + t)\right|^2 +
\left|H'(-x + t)\right|^2 \right] \, dx
$$
$$
+ \int_{0}^{a(t)} \frac{1}{2}[V_{\xi} V_{\tau}] \left[\left|H'(x +
t)\right|^2 - \left|H'(-x + t)\right|^2 \right] \, dx.
$$
Also, differentiating $x = ( H^{-1} (\xi + \tau) - H^{-1}(- \xi +
\tau))/ 2$, we obtain
$$dx = 1/2 ((H^{-1})' (\xi +
\tau) + (H^{-1})'(- \xi + \tau)) d\xi.
$$
The inequality $\left|V_{\xi} V_{\tau}\right| \leq 1/2
\left(V_{\xi}^2 + V_{\tau}^2\right)$ yields to:
\begin{equation}\label{inh}
C_1 E_V(\tau)\leq E_u(t)\leq  C_2 E_V(\tau),
\end{equation}
for positive constants $C_1, C_2$.
\end{proof}
\begin{rmk}
Applying the transformation $\Phi$, the system (\ref{B.Khiar})-(\ref{dir}) becomes:
\begin{equation}
\label{Beni.Khalled} 
\left\{\begin{array}{lll} 
\partial_{\tau}^{2}V - \partial_{\xi}^{2} V = 0, \;\;\;\; \mbox{for} \;\;\;\; 0<\xi<\rho(F) /2, \, \tau \in \mathbb{R}, \\
V(0,\tau) = 0, \;\; V(\rho(F) /2,\tau) = 0, \;\; \tau \in \R, \\
V(\xi,0) = \phi_1(\xi), \;\; V_{\tau}(\xi,0) = \psi_1(\xi),
\;\; \xi \in (0,\rho(F) /2). 
\end{array}
\right.
\end{equation}
\end{rmk} We need the following Lemma.
\begin{lemma}\label{Ob1}
If $T>\rho(F) $, there exists $C(T)>0$ such that for all
$(\phi_1,\psi_1) \in H^1_0(0,\rho(F)/2) \times L^2(0,\rho(F)/2)$ we
have
$$
C(T) \, \int_{0}^{T} |V_{\xi}(\rho(F) /2,\tau)|^2d\tau \geq
\|\phi_1\|^2_{H^1_0(0,\rho(F) /2)} + \|\psi_1\|^2_ {L^2(0,\rho(F)
/2)}.
$$
\end{lemma}
\begin{proof}[\bf{Proof of Theorem~\ref{dirichlet}}]

We consider (\ref{B.Khiar})-(\ref{dir}) and state:
$$
\partial_{x}u = \partial_{\xi}V \partial_{x}\xi
+ \partial_{\tau}V \partial_{x}\tau
$$
Next we have:
$$
\partial_{x}u(a(t),t)= \partial_{\xi}V(\rho(F) /2,\tau)
\partial_{x}\xi(a(t),t)+\partial_{\tau}V(\rho(F) /2,\tau)
\partial_{x}\tau(a(t),t)
$$
Since
$$
\partial_{x}\xi = [H'(x + t)+ H'(-x + t)]/ 2 \;\; \mbox{and}
\;\; \partial_{x}\tau = [H'(x + t)- H'(-x + t)]/ 2,
$$
it follows that:
$$
|\partial_{x}u(a(t),t)|^2= |\partial_{\xi}V(\rho(F) /2,\tau)
\partial_{x}\xi(a(t),t)+\partial_{\tau}V(\rho(F) /2,\tau)
\partial_{x}\tau(a(t),t)|^2
$$
$$
=\frac{1}{4} \left\{|\partial_{\xi}V(\rho(F) /2,\tau)[H'(x + t)+
H'(-x + t)]\right\}^2.
$$
Make use of Young inequalities, (\ref{Obserneu}) is a consequence of
the inequalities (\ref{l1l2}) and (\ref{l2l1}), Lemma \ref{second}
and Lemma \ref{Ob1}.
\end{proof}

\begin{proof}[\bf{Proof of Corollary~\ref{cor}}]
Let us consider 
\begin{equation}\label{cont1}
\left\{\begin{array}{lll} \partial_{\tau}^{2}V -
\partial_{\xi}^{2}V = 0, \;\;\;\; \mbox{for} \;\;\;\; 0<\xi<\rho(F) /2, \, \tau \in \mathbb{R}, \\
V(0,\tau) = 0, \;\; V(\rho(F) /2,\tau)= 0, \;\; \tau \in \R, \\
V(\xi,0) = V_0(\xi), \;\; V_{\tau}(\xi,0) = V_1(\xi),
\;\; \xi \in (0,\rho(F)/2). \\
\end{array}\right.
\end{equation}
The system (\ref{cont1}) is exactly observable at time $\rho(F)$ that is: there exists $C>0$ such
that for all $\tau \geq \rho(F)$, we have
$$
C(T) \, \int_{0}^{T} |V_{\xi}(\rho(F) /2,\tau)|^2d\tau \geq
\|\phi_1\|^2_{H^1_0(0,\rho(F) /2)} + \|\psi_1\|^2_ {L^2(0,\rho(F)
/2)},
$$
and so the following problem
\begin{equation}\label{cont2}
\left\{\begin{array}{lll} \partial_{\tau}^{2}\tilde{V} -
\partial_{\xi}^{2}\tilde{V} = 0, \;\;\;\; \;\;\;\; 0<\xi<\rho(F) /2, \, \tau \in \mathbb{R}, \\
\tilde{V}(0,\tau) = 0, \;\; \tilde{V}(\rho(F) /2,\tau)= g(\tau), \;\; \tau \in \R, \\
\tilde{V}(\xi,0) = \tilde{V}_0(\xi), \;\; \tilde{V}_{\tau}(\xi,0) =
\tilde{V}_1(\xi),
\;\; \xi \in (0,\rho(F)/2)\\
\end{array}\right.
\end{equation}
is exactly controllable at $\rho(F)$ that is for all  $(\tilde{V}_0,
\tilde{V}_1) \in L^2(0,\rho(F) /2)\times H^{-1}(0,\rho(F) /2)$,
there exists $g \in L^2(0,\rho(F))$ such that $\tilde{V}(\xi,\tau)=0$ for all $\tau \geq \rho(F)$.\\ Moreover, $g:=V_{\xi}(\rho(F) /2,\tau)
\chi_{(0,\rho(F))}(\tau)$.

\smallskip 

So the following transformed system
\begin{equation*}
\left\{\begin{array}{lll} \partial_{t}^{2}u -
\partial_{x}^{2}u = 0, \;\;\;\; \;\;\;\; 0<x<a(t), \, t \in \mathbb{R},\\
u(0,t) = 0, \;\; u(a(t),t)= f(t), \;\; t \in \R, \\
u(x,0) = u_0(x), \;\; u_{t}(x,0) =u_1(x),
\;\; x \in (0,a(0))\\
\end{array}\right.
\end{equation*}
is exactly controllable with a time of control $\displaystyle{
T:=|e^{\frac{\rho(F)-h_2}{h_0}}-h_1|}$ and a control $f(t)$ is given by $f(t)= g \left(\frac{H(a(t)+t)+H(-a(t)+t)}{2} \right).$
\end{proof}

\section{Further comments: The quasi periodic case}
One can try to generalize the previous results to the case when $a$
is no longer periodic but has some sort of quasiperiodicity
\footnote{A function $a(t), t \in \R$ is called quasiperiodic with
basic frequencies $\omega = (\omega_1 , ..., \omega_m) \in \R^{m}$
(briefly $2\pi/\omega-$q.p) if there exists a continuous function
$\hat{g}(\theta), \theta = (\theta_1 , ..., \theta_m) \in \R^{m}$
that is $2\pi$-periodic in each $\theta_i , i=1, ..., m$ such that
$a(t) = \hat{a}(\omega t)$ holds. $\hat{g}(\theta)$, is called the
corresponding function and $\displaystyle\frac{2\pi}{\omega} =
(\displaystyle\frac{2\pi}{\omega_1} , ...,
\displaystyle\frac{2\pi}{\omega_m})$ the basic periods of $a$.}.\\
The problem is much more complicated, since there is no rotation
number. However, in \cite{Ya 5} the author uses a weaker notion of
upper (resp. lower) rotation number of $F$ at every point $x$ as
follows:
$$
\overline{\rho}(F) = \displaystyle\limsup_{n\to
+\infty}\frac{F^n(x)-x}{n}
$$
$$ \mbox{ (resp. }
\underline{\rho}(F) =
\displaystyle\liminf_{n\to+\infty}\frac{F^n(x)-x}{n} ).
$$
As a consequence, it is shown that under the same Diophantine
condition \cite{Ca}, \cite{Lan} satisfied by $\overline{\rho}(F)$
(resp. $\underline{\rho}(F)$), the rotation number of $F$ exists and
coincides with the lower (resp. upper) rotation number.
\begin{lemma}\label{reduction}
Assume that $a(t)$ is an $\eta-$q.p function, $\hat{a}(\theta)$ is
real analytic and satisfy $|\hat{a'}(\theta)|<1$ for $\eta, \theta
\in \R^m$ and set $\beta = (\displaystyle\frac{2\pi}{\eta_1} ,
...,\displaystyle\frac{2\pi}{\eta_m})$. Assume also that there
exists $C_0>0$ depending on $\beta$ such that $|(k,\beta) + \pi
l/\overline{\rho}(F)|
> \displaystyle\frac{C_0}{|k|^{m+1}}$. Then, there exists a real
analytic function $H(\xi) = \xi+h(\xi)$, where $h(\xi)$ is an
$\eta$-q.p. function, such that
\begin{equation}\label{red1}
H^{-1} \circ F \circ H(\xi) = \xi + \overline{\rho}(F).
\end{equation}
\end{lemma}
\begin{rmk}
Thanks to Lemma \ref{reduction}, Theorem \ref{dirichlet}
is easily extended by similar arguments.
\end{rmk}
\begin{rmk}
Generalizations of the foregoing results may be obtained in a 3D
context, assuming that solutions and given data are functions of
only $r=(x^2+y^2+z^2)^{1/2}$ with respect to the space variables.
Let $\Omega$ be the domain $0<r<a(t)$ and consider,
\begin{equation}
\left\{\begin{array}{lll} u_{tt}-u_{rr}-(2/r)u_r = 0 \;\;\;\;
\mbox{in} \;\;\;\; {\Omega}, \, t > 0,\\
\mbox{with boundary conditions} \;\;\;\; u(0,t)=u(a(t),t)=0, \, t > 0,\\
\mbox{and initial conditions} \;\;\;\; u(r,0)=\phi(r), u_t(r,0)
=\psi(r), \,  0 <r < a(0).
\end{array}
\right.
\end{equation}
Introducing the transformation $u(r,t)=w(r,t)/r$ leads to the
problem:
\begin{equation}
\left\{\begin{array}{lll} w_{tt}=w_{rr} \;\;\;\; 0< r<a(t), \, t > 0,\\
w(0,t)=w(a(t),t)=0, \, t > 0, \\\\
w(r,0)=r\phi(r), \;\; w_t(r,0)= r\psi(r), \, 0 < r < a(0). 
\end{array}\right.
\end{equation}
\end{rmk}
\section{Appendix}
In this section, we treat the Dirichlet observability.
\begin{thm}[Dirichlet observability]\label{Neumann}
Under the assumptions \ref{a1} and \ref{a2}, suppose moreover that $l_1<l_2,$ there exist $ T, \;
C^{*}> 0$ such that for all solution $u$ of the system
(\ref{B.Khiar}) with the mixed boundary condition (\ref{neu}) and
initial data $(\phi,\psi) \in H^1_l(0,a(0)) \times L^2(0,a(0)),$ we
have
\begin{equation}\label{Obser}
\int_{0}^{T} \left|u_{t}(a(t),t) \right|^2 \, dt \geq C^{*} \,
\left(\|\phi\|^2_{H^1_l(0,a(0))} + \|\psi\|^2_ {L^2(0,a(0))} \right).
\end{equation}
\end{thm}
\begin{rmk}
Using $\Phi$ given by (\ref{trans}), we transform the system (\ref{B.Khiar})-(\ref{neu})
into:
\begin{equation}\label{Beni.Benou}
\left\{\begin{array}{lll} \partial_{\tau}^{2}V -
\partial_{\xi}^{2}V = 0, \;\;\;\; \mbox{for} \;\;\;\; 0<\xi<\rho(F) /2, \, \tau \in \mathbb{R},\\
V(0,\tau) = 0, \;\; V_{\xi}(\rho(F) /2,\tau)+b(t(\tau))V_{\tau}(\rho(F) /2,\tau) = 0, \;\; \tau \in \R \\
V(\xi,0) = \phi_2(\xi), \;\; V_{\tau}(\xi,0) = \psi_2(\xi),
\;\; \xi \in (0,\omega/2). 
\end{array}\right.
\end{equation}

\end{rmk}
For the proof of Theorem~\ref{Neumann}, we need the following
lemmas.
\begin{lemma}\label{exv}
Assume that $l_1<l_2$, then there exist positive constants $C$ and
$\omega$ such that
\begin{equation}\label{expv}
E_{V}(\tau)\leq C e^{-\omega \tau} E_V(0).
\end{equation}
\end{lemma}
\begin{proof}
Define the Lyapunov function:
$$E_1(\tau) = \frac{1}{2}\int^{\rho(F)}_0[V_{\xi}^2(\xi,\tau)+ V_{\tau}^2(\xi,\tau)]d\xi+
\delta\int^{\rho(F)}_0 \xi
V_{\xi}(\xi,\tau)V_{\tau}(\xi,\tau)d\xi.$$ We obtain for
$\delta<\frac{1}{\rho(F)},$
\begin{equation}\label{E1E}
0<(1 -\delta\rho(F))E_{V}(\tau) \leq E_1(\tau) \leq (1
+\delta\rho(F))E_{V}(\tau).
\end{equation}
We derive $E_1$ with respect to $\tau$, we get
\begin{eqnarray*}
E_1'(\tau)&=&
[V_{\xi}V_{\tau}]^{\xi=\rho(F)}_{\xi=0}-\frac{\delta}{2}\int^{\rho(F)}_0
[V_{\xi}^2(\xi,\tau)+ V_{\tau}^2(\xi,\tau)]d\xi+\frac{\delta}{2}[\xi
(V_{\xi}^2 +
V_{\tau}^2)]^{\xi=\rho(F)}_{\xi=0}\\
&=&[\frac{\delta}{2}(1+b(t(\tau))^2)-b(t(\tau))]V_{\tau}^2(\rho(F),\tau)-
\frac{\delta}{2}\int^{\rho(F)}_0[V_{\xi}^2(\xi,\tau)+
V_{\tau}^2(\xi,\tau)]d\xi.
\end{eqnarray*}
We choose $\delta$ small enough, taking into account (\ref{btau})
and (\ref{E1E}) we get
$$ E_1'(\tau)\leq -\omega E_1(\tau). $$ The proof is complete.
\end{proof}
\begin{lemma}\label{Obser1bb}
If $T>\rho(F)$, then there exists $C(T)>0$ such that for all
$(\phi_2,\psi_2) \in H_l^1(0,\rho(F)/2) \times L^2(0,\rho(F)/2)$ we
have
\begin{equation}\label{obes1}
C(T) \, \int_{0}^{T} \left|V_{\xi}(\rho(F) /2,\tau)\right|^2 \,
d\tau \geq \|\phi_2\|^2_{H_l^1(0,\rho(F)/2)} +
\|\psi_2\|^2_{L^2(0,\rho(F) /2)},
\end{equation}
and
\begin{equation}\label{obes2}
C(T) \, \int_{0}^{T} \left|V_{\tau}(\rho(F) /2,\tau)\right|^2 \,
d\tau \geq \|\phi_2\|^2_{H_l^1(0,\rho(F)/2)} +
\|\psi_2\|^2_{L^2(0,\rho(F) /2)}.
\end{equation}
\end{lemma}
\begin{proof}
The energy identity for the system (\ref{Beni.Benou}) gives :
$$ E_V(T)-E_V(0)=-\int_{o}^{T}b(t(\tau))|V_{\tau}(\rho(F)/2,\tau)|^2d\tau.$$
Using (\ref{btau}) and (\ref{expv}), we obtain
\begin{eqnarray*}
\int_{0}^{T} \left|V_{\tau}(\rho(F) /2,\tau)\right|^2 \, d\tau &\geq
&
C \int_{0}^{T} b(t(\tau))\left|V_{\tau}(\rho(F) /2,\tau)\right|^2 \, d\tau \\
&\geq & C( E_V(0)- E_V(T))\\
&\geq & CE_V(0)(1- e^{-\omega T}).
\end{eqnarray*}
This permit to conclude the second inequality in Lemma \ref{Obser1bb}. \\
For the first inequality, it suffices to use (\ref{Beni.Benou}) and
(\ref{btau}).
\end{proof}
\begin{proof}[\bf{Proof of Theorem~\ref{Neumann}}]
For the proof of (\ref{Obser}), we state as above:
$$
\partial_{t}u = \partial_{\xi}V \partial_{t}\xi
+ \partial_{\tau}V \partial_{t}\tau.
$$
Next we have:
\begin{eqnarray*}
|\partial_{t}u(a(t),t)|^2
 &=&\frac{1}{4} \{|\partial_{\xi}V(\rho(F) /2,\tau)[H'(x + t) - H'(-x +
t)]
\\
&+&\partial_{\tau}V(\rho(F) /2,\tau) [H'(x + t) + H'(-x + t)]\}^2.
\end{eqnarray*}
Make use of Young inequalities, Lemma \ref{second}, (\ref{obes1}),
(\ref{obes2}) and (\ref{l1l2}), we obtain the desired result.
\end{proof}

\end{document}